\documentclass[12pt]{article}
\usepackage{amssymb}
\usepackage{amsfonts}
\usepackage{amsthm}
\usepackage{amsmath}
\usepackage{fullpage}

\newtheorem{Lemma}{Lemma}

\newtheorem{Conjecture}{Conjecture}
\newtheorem{Theorem}{Theorem}

\newtheorem{Definition}{Definition}

\newcommand{\conv}{\text{conv}}
\newcommand{\sep}{\text{sep}}
\newcommand{\Int}{\text{int}}

\usepackage[pdftex]{graphicx}
\graphicspath{{C:/Users/Sam/Desktop/PaperFigures/}}
\DeclareGraphicsExtensions{.pdf,.png}

\begin{document}
\title{Regular Totally Separable Sphere Packings}
\author{Samuel Reid\thanks{University of Calgary, Centre for Computational and Discrete Geometry (Department of Mathematics \& Statistics), and Thangadurai Group (Department of Chemistry), Calgary, AB, Canada. $\mathsf{e-mail: smrei@ucalgary.ca}$.}}
\maketitle

\begin{abstract}
The topic of totally separable sphere packings is surveyed with a focus on regular constructions, uniform tilings, and contact number problems.
An enumeration of all regular totally separable sphere packings in $\mathbb{R}^2$, $\mathbb{R}^3$, and $\mathbb{R}^4$ which are based on
convex uniform tessellations, honeycombs, and tetracombs, respectively, is presented, as well as a construction
of a family of regular totally separable sphere packings in $\mathbb{R}^d$
that is not based on a convex uniform $d$-honeycomb for $d \geq 3$.
\end{abstract}

\textbf{Keywords:} sphere packings, hyperplane arrangements, contact numbers, separability. \\
\text{   \;\;       } \textbf{MSC 2010 Subject Classifications:} Primary 52B20, Secondary 14H52.

\section{Introduction}
In the 1940s, P. Erd\"{o}s introduced the notion of a separable set of domains in the plane, which gained
the attention of H. Hadwiger in \cite{Hadwiger}.
G.F. T\'{o}th and L.F. T\'{o}th extended this notion to totally separable domains and proved the densest totally
separable arrangement of congruent copies of a domain is given by a lattice packing of the domains generated by the
side-vectors of a parallelogram of least area containing a domain \cite{Toth}. Totally separable
domains are also mentioned by G. Kert\'{e}sz in \cite{Kertesz}, where it is proved that a cube of volume $V$ contains a totally separable set of $N$ balls of radius $r$
with $V \geq 8Nr^3$. Further results and references regarding separability can be found in a
manuscript of J. Pach and G. Tardos \cite{Pach}.

This manuscript continues the study of separability in the context of regular unit sphere packings, i.e., infinite
sets of unit spheres $$\mathcal{P} = \bigcup_{i=1}^{\infty} \left(x_{i} + \mathbb{S}^{d-1}\right)$$ 
in $\mathbb{R}^d$ with $\| x_{i} - x_{j} \| \geq 2$, whose contact graphs $G_{\mathcal{P}} = (V,E)$, where $V = \{x_{i} \; | \; i \in \mathbb{N}\}$ and
$$E = \left\{(i,j) \; | \; \left(x_{i} + \mathbb{S}^{d-1}\right)\cap \left(x_{j} + \mathbb{S}^{d-1}\right) \neq \emptyset\right\},$$
are regular (every vertex has equal degree); this means that every sphere in the packing touches the same number
of spheres.

Let $C(\mathcal{P}_{n})$ be the contact number of a unit sphere packing $\mathcal{P}_{n}$ with $n$ spheres, i.e.,
the cardinality of the edge set of the contact graph $G_{\mathcal{P}_{n}}$. Determining the maximum contact number of
a unit sphere packing with $n$ spheres is known as the contact number problem. The contact number problem for circle packings
in $\mathbb{R}^2$ was solved exactly in 1974 by H. Harborth in \cite{Harborth} to be $\lfloor 3n - \sqrt{12n -3} \rfloor.$
Upper and lower bounds on the contact number problem for finite packing of unit balls
in $\mathbb{R}^3$ were provided by K. Bezdek and the author in \cite{Reid} and studied in detail up to $n=18$ by M. Holmes-Cerfon in \cite{Holmes-Cerfon}
improving the lower bounds for some values. Consult \cite{Bezdek} and references therein for more information regarding 
contact numbers of unit sphere packings and arrangements of spheres in higher dimensions.

\begin{Definition}
A sphere packing $\mathcal{P}$ is totally separable if every tangent hyperplane to a pair of touching spheres has
an empty intersection with the interior of all spheres in $\mathcal{P}$.
\end{Definition}

The contact number problem for totally separable sphere packings is studied and all regular totally separable
 sphere packings in $\mathbb{R}^2$, $\mathbb{R}^3$, and $\mathbb{R}^4$ based on convex uniform tessellations (classified in an
 unpublished manuscript of G. Olshevsky \cite{Olshevsky}) are constructed. Now, let
$$c(n,d) = \max_{\sep(\mathcal{P}_{n})=1} C(\mathcal{P}_{n}),$$ where $\sep(\cdot)$ is a measure on sphere packings called the
\textit{separability} of the packing which is defined formally in the appendix; intuitively, the separability of a packing is
0 if the packing is inseparable and 1 if it is totally separable. The theory of minimal area polyominoes developed in \cite{Alonso}
is used with Euler's formula to provide a proof of the contact number problem for totally
separable circle packings:
$$c(n,2) = \Big \lfloor 2(n - \sqrt{n}) \Big \rfloor.$$
Furthermore, heuristics are provided for the upper bound on the contact number problem for totally separable sphere packings in $\mathbb{R}^d$
which is based on the number of edges of polyominoes over the cubic $d$-honeycomb and hence exact when $\sqrt[d]{n} \in \mathbb{N}$:
$$c(n,d) \leq \Big\lfloor d \left(n - n^{\frac{d-1}{d}}\right) \Big\rfloor.$$
As this manuscript was being prepared, K. Bezdek, B. Szalkai, and I. Szalkai proved the above upper bound on
$c(n,d)$ with an ingenious argument involving box-polytopes and the isoperimetric inequality \cite{BezdekSz}.
The paper ends with a construction of a family of regular totally separable sphere packings in $\mathbb{R}^d$
that is not based on a convex uniform tessellation for $d \geq 3$ and an outline of future research directions.

The most basic example of when the condition on a totally separable sphere packing is violated is explained
in the form of a lemma for future reference.

\begin{Lemma}\label{triangle}
If the contact graph $G_{\mathcal{P}}$ of a sphere packing $\mathcal{P}$ in $\mathbb{R}^d$ contains a $k$-simplex
for $2 \leq k \leq d$, then $\mathcal{P}$ is not totally separable.
\end{Lemma}
\begin{proof}
First consider the case where $G_{\mathcal{P}}$ contains a $2$-simplex and observe that it violates
total separability. For, the tangent line generated by the touching circles associated with an edge $e$ of
the $2$-simplex intersects the interior of the circle associated with the vertex which is not an endpoint of $e$.
Proceed by induction, observing from the base case $d=2$ that any $k$-simplex with $3 \leq k \leq d$ 
in $G_{\mathcal{P}}$ violates total separability as that $k$-simplex contains
a $2$-simplex somewhere in its flag, thus proving the lemma.
\end{proof}

This lemma will be used extensively for classifying totally separable sphere packings based on convex
uniform tesselations of $\mathbb{R}^d$, also known as tilings or honeycombs.

\section{Regular Totally Separable Circle Packings in $\mathbb{R}^2$}
Regular totally separable circle packings in $\mathbb{R}^2$ which are based on convex uniform tilings are classified
by the following theorem.
\begin{Theorem}\label{tiling2}
There are exactly 4 convex uniform tilings in $\mathbb{R}^2$ which generate totally separable
circle packings:
\begin{enumerate}
 \item $\mathcal{P}$1 - Square tiling, $\{4,4,4\}$
 \item $\mathcal{P}$3 - Hexagonal tiling, $\{6,6,6\}$
 \item $\mathcal{K}$6 - Truncated square tiling, $\{4,8,8\}$
 \item $\mathcal{K}$9 - Omnitruncated trihexagonal tiling, $\{4,6,12\}$
\end{enumerate}
\end{Theorem}
\begin{proof}
Apply Lemma \ref{triangle} to the list of 11 convex uniform tilings of $\mathbb{R}^2$; three Pythagorean
tilings and eight Keplerian tilings \cite{Olshevsky}. Clearly, if $\mathcal{P}$ is a $4$-regular totally separable packing of unit circles in $\mathbb{R}^2$
generated by a convex uniform tiling, 
then $\mathcal{P}$ is congruent to $\mathcal{P}$1. If $\mathcal{P}$ is a $3$-regular totally
separable packing of unit circles in $\mathbb{R}^2$ generated by a convex uniform tiling, then $\mathcal{P}$ is congruent to $\mathcal{P}3$, $\mathcal{K}6$,
$\mathcal{K}$9 or a subset of $\mathcal{P}$1. If $\mathcal{P}$ is a $2$-regular totally separable packing of unit circles in
$\mathbb{R}^2$ generated by a convex uniform tiling, then $\mathcal{P}$ is congruent to a subset of either $\mathcal{P}1$, $\mathcal{P}$3, $\mathcal{K}$6, or $\mathcal{K}$9.
\end{proof}

The theory of minimal area polyominoes and Euler's formula is used to provide an exact solution to the contact number problem for totally separable circle packings;
an alternative explicit proof, not relying on the results of \cite{Alonso}, which extends a proof
technique of H. Harborth \cite{Harborth} appears in \cite{BezdekSz}.

\begin{Theorem}
Given $n \in \mathbb{N}$, there exists a totally separable circle packing $\mathcal{P}_{n}$ in $\mathbb{R}^2$ with contact
number $$C(\mathcal{P}_{n}) = \Big\lfloor 2(n - \sqrt{n})\Big\rfloor.$$
Furthermore, no totally separable circle packing in $\mathbb{R}^2$ has a larger contact number.
\end{Theorem}
\begin{proof}
By Euler's formula, $n - (|E| + P(c)) + a = 2$, where $|E|$ is the cardinality of the edge set of the contact graph $G_{\mathcal{P}_{n}}$,
$P(c)$ is the perimeter of the polyomino $c$ with area $a$ generated by placing $n$ unit $2$-cubes so that elements of $\mathcal{P}_{n}$ are incircles.
Interpolate the piece-wise defined function from Corollary 2.5 of \cite{Alonso} which provides the minimal perimeter of a polyomino of area $a$ in order to
obtain the desired formula.
\end{proof}

\begin{figure}[h!]
\begin{center}
\includegraphics[scale=0.7]{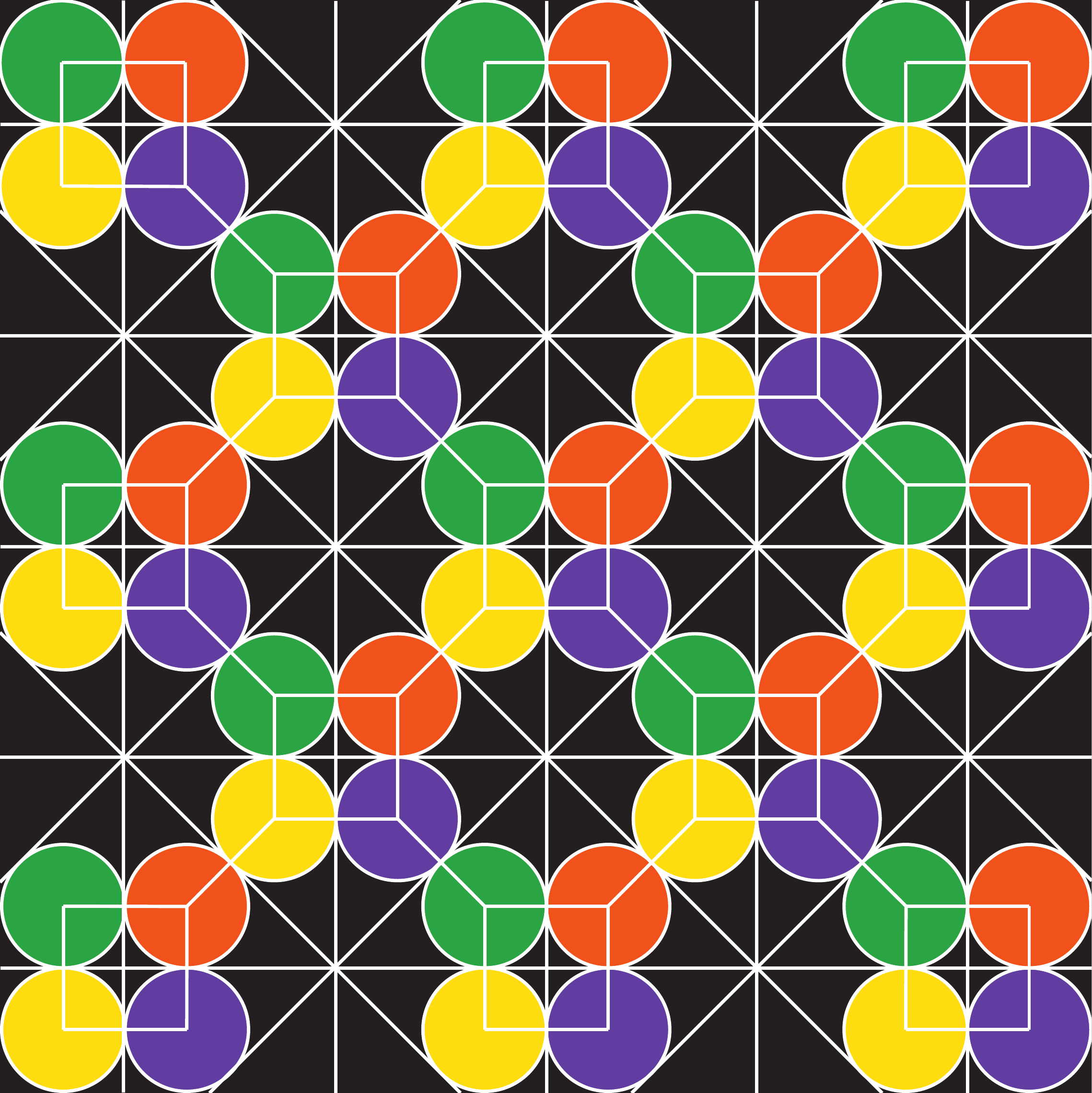}
\caption{A finite part of the contact graph, convex uniform tiling, and $3$-regular totally separable circle packing generated by the truncated square tiling.}
\end{center}
\end{figure}

\section{Regular Totally Separable Sphere Packings in $\mathbb{R}^3$}
Regular totally separable sphere packings in $\mathbb{R}^3$ which are based on convex uniform honeycombs are
classified by the following theorem.
\begin{Theorem}
There are exactly 7 convex uniform honeycombs in $\mathbb{R}^3$ which generate totally separable
sphere packings in $\mathbb{R}^3$:
\begin{enumerate}
 \item $\mathcal{J}$1 - Cubic honeycomb
 \item $\mathcal{J}$3 - Hexagonal prismatic honeycomb
 \item $\mathcal{J}$6 - Truncated square prismatic honeycomb
 \item $\mathcal{J}$9 - Omnitruncated trihexagonal prismatic honeycomb
 \item $\mathcal{J}$16 - Bitruncated cubic honeycomb
 \item $\mathcal{J}$18 - Cantitruncated cubic honeycomb
 \item $\mathcal{J}$20 - Omnitruncated cubic honeycomb
\end{enumerate}
\end{Theorem}
\begin{proof}
Apply Lemma \ref{triangle} to N. Johnson's list of 28 convex uniform honeycombs \cite{Johnson}. Clearly, if $\mathcal{P}$
is a $6$-regular totally separable packing of unit spheres in $\mathbb{R}^3$ generated by a convex
uniform honeycomb, then
$\mathcal{P}$ is congruent to $\mathcal{J}$1. If $\mathcal{P}$ is a $5$-regular totally separable packing of unit spheres
in $\mathbb{R}^3$ generated by a convex
uniform honeycomb, then $\mathcal{P}$ is congruent to $\mathcal{J}$3, $\mathcal{J}6$, $\mathcal{J}$9, or a subset of $\mathcal{J}1$.
If $\mathcal{P}$ is a $4$-regular totally separable packing of unit spheres in $\mathbb{R}^3$ generated by a convex
uniform honeycomb, then $\mathcal{P}$ is congruent
to $\mathcal{J}16$, $\mathcal{J}18$, $\mathcal{J}20$, or a subset of either $\mathcal{J}1$, $\mathcal{J}$3, $\mathcal{J}$6, or $\mathcal{J}$9. 
If $\mathcal{P}$ is a $3$-regular, or $2$-regular totally 
separable packing of unit spheres in $\mathbb{R}^3$ generated by a convex
uniform honeycomb, then $\mathcal{P}$ is congruent to a subset of either $\mathcal{J}$1,
$\mathcal{J}$3, $\mathcal{J}$6, $\mathcal{J}$9, $\mathcal{J}$16, $\mathcal{J}$18, or $\mathcal{J}$20.
\end{proof}

\section{Regular Totally Separable Sphere Packings in $\mathbb{R}^4$}
Regular totally separable sphere packings in $\mathbb{R}^4$ based on convex uniform 4-honeycombs are
classified by the following theorem.

\begin{Theorem}\label{R4}
There are exactly $18$ convex uniform tetracombs in $\mathbb{R}^4$ which generate totally separable
sphere packings in $\mathbb{R}^4$:
\begin{enumerate}
 \item $\mathcal{O}$1 - Tesseractic tetracomb
 \item $\mathcal{O}$3 - Square-hexagonal duoprismatic tetracomb
 \item $\mathcal{O}$6 - Tomosquare-square duoprismatic tetracomb
 \item $\mathcal{O}$9 - Omnitruncated-trihexagonal-square duoprismatic tetracomb
 \item $\mathcal{O}$16 - Bitruncated-cubic prismatic tetracomb
 \item $\mathcal{O}$18 - Cantitruncated-cubic prismatic tetracomb
 \item $\mathcal{O}$20 - Omnitruncated-cubic prismatic tetracomb
 \item $\mathcal{O}$39 - Hexagonal duoprismatic tetracomb
 \item $\mathcal{O}$42 - Hexagonal-tomosquare duoprismatic tetracomb
 \item $\mathcal{O}$45 - Hexagonal-omnitruncated-trihexagonal duoprismatic tetracomb
 \item $\mathcal{O}$63 - Tomosquare duoprismatic tetracomb
 \item $\mathcal{O}$66 - Tomosquare-omnitruncated-trihexagonal duoprismatic tetracomb
 \item $\mathcal{O}$78 - Omnitruncated-trihexagonal duoprismatic tetracomb
 \item $\mathcal{O}$99 - Truncated icositetrachoric tetracomb
 \item $\mathcal{O}$100 - Great diprismatotesseractic tetracomb
 \item $\mathcal{O}$103 - Omnitruncated tesseractic tetracomb
 \item $\mathcal{O}$132 - Omnitruncated icositetrachoric tetracomb
 \item $\mathcal{O}$140 - Great-prismatodecachoric tetracomb
\end{enumerate}
\end{Theorem}
\begin{proof}
Apply Lemma \ref{triangle} to G. Olshevsky's list of 143 convex uniform 4-honeycombs \cite{Olshevsky}.
Clearly, if $\mathcal{P}$ is a $8$-regular totally separable packing of unit spheres in $\mathbb{R}^4$ generated by a convex
uniform tetracomb, then
$\mathcal{P}$ is congruent to $\mathcal{O}$1. If $\mathcal{P}$ is a $7$-regular totally separable packing of unit spheres
in $\mathbb{R}^4$ generated by a convex
uniform tetracomb, then $\mathcal{P}$ is congruent to $\mathcal{O}$3, $\mathcal{O}$6, $\mathcal{O}$9, or a subset of
$\mathcal{O}$1. If $\mathcal{P}$ is a $6$-regular totally separable packing of unit spheres in $\mathbb{R}^4$ generated by a convex
uniform tetracomb, then $\mathcal{P}$ is congruent
to $\mathcal{O}$16, $\mathcal{O}$18, $\mathcal{O}$20, $\mathcal{O}$39, $\mathcal{O}$42, $\mathcal{O}$45, $\mathcal{O}$63,
$\mathcal{O}$66, $\mathcal{O}$78, or a subset of either $\mathcal{O}$1, $\mathcal{O}$3, $\mathcal{O}$6, or $\mathcal{O}$9.
If $\mathcal{P}$ is a $5$-regular totally separable packing of unit spheres in $\mathbb{R}^4$ generated by a convex
uniform tetracomb, then $\mathcal{P}$ is congruent
to $\mathcal{O}$99, $\mathcal{O}$100, $\mathcal{O}$103, $\mathcal{O}$132, $\mathcal{O}$140, or a subset of either $\mathcal{O}$1, 
$\mathcal{O}$3, $\mathcal{O}$6, $\mathcal{O}$9, $\mathcal{O}$16, $\mathcal{O}$18, $\mathcal{O}$20, $\mathcal{O}$39,
$\mathcal{O}$42, $\mathcal{O}$45, $\mathcal{O}$63, $\mathcal{O}$66, or $\mathcal{O}$78. If $\mathcal{P}$ is a $4$-regular,
$3$-regular, or $2$-regular totally separable packing of unit spheres in $\mathbb{R}^4$ generated by a convex
uniform tetracomb, then $\mathcal{P}$ is congruent
to a subset of either $\mathcal{O}$1, $\mathcal{O}$3, $\mathcal{O}$6, $\mathcal{O}$9, $\mathcal{O}$16, $\mathcal{O}$18,
$\mathcal{O}$20, $\mathcal{O}$39, $\mathcal{O}$42, $\mathcal{O}$45, $\mathcal{O}$63, $\mathcal{O}$66, $\mathcal{O}$78,
$\mathcal{O}$99, $\mathcal{O}$100, $\mathcal{O}$103, $\mathcal{O}$132, or $\mathcal{O}$140.
\end{proof}

The regularity of each $4$-honeycomb is determined by inspecting the number of vertices of the vertex figure
associated with the honeycomb, e.g., the vertex figure of $\mathcal{O}$100 is an irregular pentachoron, implying that the
4-dimensional sphere packing generated by the great diprismatotessseractic tetracomb is $5$-regular.

\section{Totally Separable Sphere Packings in $\mathbb{R}^d$}
Totally separable sphere packings in $\mathbb{R}^d$ are studied and future research directions are outlined.
The following heuristics for the upper bound to the contact number problem for totally separable sphere 
packings in $\mathbb{R}^d$ provides a reasonable intuitive explanation of the following theorem.

From the formula for the number of $m$-cubes on the boundary of a $d$-cube for $m=1$ observe that
$$2^{d-1} {d \choose 1} = \Big\lfloor d \left(2^d - (2^d)^{\frac{d-1}{d}} \right)\Big\rfloor = \Big\lfloor d \left(n - n^{\frac{d-1}{d}} \right)\Big\rfloor$$
for $n=2^{d}$. Similarly, for any $n=\sqrt[d]{k} \in \mathbb{N}$ there is a $\underbrace{k \times k \times \cdot\cdot\cdot \times k}_{\text{d times}}$
$d$-cube with $\Big\lfloor d \left(k^{d} - (k^{d})^{\frac{d-1}{d}}\right) \Big\rfloor$ edges, implying that the upper bound in the following
theorem is an equality. Assume that $k^{d} < n < (k+1)^{d}$ and observe that the upper bound on $c(n,d)$ overestimates the supremum
over edge cardinalities of $(k+\delta_{1}) \times (k + \delta_{2}) \times \cdot\cdot\cdot \times (k + \delta_{d})$ 
unit polyominoes with $n$ cells, where $\delta_{i} \in \{0,1\}$.

\begin{Theorem}
For $n \in \mathbb{N}$, $$c(n,d) \leq \Big\lfloor d \left(n - n^{\frac{d-1}{d}}\right)\Big\rfloor,$$
with equality when $\sqrt[d]{n} \in \mathbb{N}$.
\end{Theorem}
\begin{proof}
Improving upon an earlier and lengthier unpublished case analytic proof, K.~Bezdek, B.~Szalkai, and I.~Szalkai provide an elegant proof using box-polytopes and the isoperimetric inequality \cite{BezdekSz}.
\end{proof}

The classification of uniform $d$-honeycombs is incomplete, leading to great difficulty in establishing
the above characterizations of totally separable sphere packings in $d=2,3,4$ for $d \geq 5$. The ongoing
work by J. Bowers, G. Olshevsky, N. Johnson, and others of classifying uniform polyterons will soon
result in the complete classification of uniform $5$-honeycombs, and the study of uniform polypetons
generating uniform $6$-honeycombs has only recently begun. For $d \geq 7$ there
appears to be no significant work on uniform honeycombs; although R. Klitzing has classified certain uniform polytopes up to $d=8$ \cite{Klitzing}. Future research on the topic of regular totally separable sphere packings should include a comprehensive construction 
of families of $k$-regular totally separable sphere packings in $\mathbb{R}^d$ for $3 \leq k \leq 2d - 1$ and $d \geq 5$. 
These are the unknown bounds on $k$-regularity because for $k=2$ spheres can be placed along an apeirogon (infinite line with evenly spaced points) and for $k=2d$
spheres can be placed on the cubic $d$-honeycomb. For an example to motivate future research in this
direction, a construction in $\mathbb{R}^d$ of a $(d+1)$-regular totally separable sphere packing which
is not based on a convex uniform $d$-honeycomb for $d \geq 3$ is presented. A similar
construction would be desired for $3 \leq k \leq d$ and $d+2 \leq k \leq 2d-1$; regardless of whether or not it is based
on a convex uniform $d$-honeycomb.

\begin{Theorem}
There exists a $(d+1)$-regular totally separable sphere packing in $\mathbb{R}^d$ for $d \geq 3$ which is not based
on a convex uniform $d$-honeycomb.
\end{Theorem}
\begin{proof}
Let $Q_{0}^{d} = \conv\left\{x_{0,1},...,x_{0,2^{d}}\right\}$ be a unit $d$-cube in $\mathbb{R}^d$ and place $2^{d}$ unit $d$-cubes
\begin{align*}
 Q_{1}^{d} &= \conv\left\{x_{1,1},...,x_{1,2^{d}}\right\} \\
 &\vdots \\
 Q_{2^{d}}^{d} &= \conv\left\{x_{2^{d},1},...,x_{2^{d},2^{d}}\right\}
\end{align*}
so that $\|x_{0,1} - x_{1,1}\| = 2,..., \|x_{0,2^{d}} - x_{2^{d},1}\|=2$ with $x_{i,1}$ lying outside $Q_{0}^d$ along a line
emanating from the centroid of $Q_{0}^{d}$ through $x_{0,i}$ for $1 \leq i \leq 2^{d}$. Now construct
$$\mathcal{P}_{2^{d} + 4^{d}} = \bigcup_{i=1}^{2^{d} + 4^{d}} \bigcup_{j=1}^{2^{d}}\left(x_{i,j} + \mathbb{S}^{d-1}\right)$$
and iteratively place $2^{d}-1$ unit $d$-cubes diagonally out of each existing unit $d$-cube $Q_{1}^{d},..., Q_{2^{d}}^{d}$
as above so that spheres may be placed around their vertices which generate a packing congruent to $\mathcal{P}_{2^d + 4^d}$.
Indefinitely extending this procedure leads to an infinite totally separable sphere packing which is $(d+1)$-regular. For, let
$x + \mathbb{S}^{d-1}$ be an arbitrary sphere in this packing and observe that it touches $d$ other spheres placed on adjacent
vertices of the unit $d$-cube which $x$ is a vertex of, and also touches $1$ other sphere which is diagonally outward as in
the construction. Furthermore, for $d=2$ this construction corresponds to the truncated square tiling $\mathcal{K}$6
and for $d \geq 3$ this construction corresponds to a scaliform which contains an elongated cubic bifrustum.
\end{proof}

The classification of regular totally separable sphere packings which are not based on convex uniform $3$-honeycombs is then
a sub-problem of classifying all scaliforms (vertex-transitive honeycombs) in $\mathbb{R}^3$; from a simplex-free scaliform
in $\mathbb{R}^3$ one can construct a totally separable sphere packing by placing equal size spheres at the vertices. The 
questionable existence of aperiodic totally separable sphere packings in any dimension remains unexplored.

\begin{Conjecture}
No aperiodic totally separable sphere packing exists in any dimension.
\end{Conjecture}

\section*{Appendix: Separability as a Geometric Measure}
Separability is introduced as a geometric measure where inseparable sphere packings have a separability of $0$ and 
totally separable sphere packings have a separability of $1$. Let $H_{e}$ denote the tangent hyperplane to a pair of
touching spheres in $\mathbb{R}^d$ associated with edge $e$ of the contact graph $G_{\mathcal{P}} = (V,E)$. First
define the separability measure for finite sphere packings $\mathcal{P}_{n}$ with $G_{\mathcal{P}_{n}} = (V_{n},E_{n})$ by
$$\sep(\mathcal{P}_{n}) = \sum_{e \in E_{n}} \frac{\left|\left\{H_{e} \; | \; H_{e} \cap \Int \left(x_{i} + \mathbb{S}^{d-1}\right) = \emptyset, 1 \leq i \leq n \right\}\right|}{|E_{n}|}.$$
If a sphere packing $\mathcal{P} \hookrightarrow \mathbb{R}^d$ can be constructed so that $\mathcal{P} = \lim_{n \rightarrow \infty} \mathcal{P}_{n}$
for some sequence of finite sphere packings $\mathcal{P}_{n}$, then
$$\sep(\mathcal{P}) = \lim_{n \rightarrow \infty} \sum_{e \in E_{n}} \frac{\left|\left\{H_{e} \; | \; H_{e} \cap \Int \left(x_{i} + \mathbb{S}^{d-1}\right) = \emptyset, 1 \leq i \leq n \right\}\right|}{|E_{n}|}.$$
Observe that if every tangent hyperplane $H_{e}$ at a contact point associated with the edge $e$ intersects the interior of
another sphere in the packing $\mathcal{P}$ then $\sep(\mathcal{P})=0$ and similarly if none intersect the interior of
a sphere in the packing then $\sep(\mathcal{P})=1$; in the former case $\mathcal{P}$ is called inseparable and in the latter
case $\mathcal{P}$ is called totally separable.

\section{Acknowledgements}
Many thanks are to my first supervisor K\'{a}roly Bezdek for introducing me to the topic of totally
separable sphere packings and having so many discussions with me regarding geometry research over the years.
Special thanks are also to Jonathan Bowers for helping to check through George Olshevsky's 143 honeycombs for 
realizable packings; three of which remained unnoticed to me in the compilation of Theorem \ref{R4}.

\end{document}